\documentclass[12pt]{amsart}
\usepackage{graphicx}
\usepackage{pinlabel}
\usepackage{color}
\usepackage[T1]{fontenc}
\usepackage[latin9]{inputenc}
\usepackage{fullpage}

\theoremstyle{definition} 
\newtheorem{definition}{Definition}[section]

\newtheorem{move}{Move}

\theoremstyle{plain} 
\newtheorem{proposition}[definition]{Proposition}
\newtheorem{lemma}[definition]{Lemma}
\newtheorem{theorem}[definition]{Theorem}

\theoremstyle{remark} 
\newtheorem{remark}[definition]{Remark}

\newtheorem{example}[definition]{Example}

\newtheorem*{Reader's guide}{Reader 's guide}

\newcommand{\HH}{\mathcal{H}}

\title{Labeled Rauzy classes and framed translation surfaces}
\author{Corentin Boissy}
\address{
Université Aix-Marseille 3
LATP, case cour A,
Faculté de Saint Jérôme
Avenue Escadrille Normandie-Niemen,
13397 Marseille cedex 20 }
\email{corentin.boissy@latp.univ-mrs.fr}


\date{\today}

\begin{document}
\begin{abstract}
In this paper, we compare two definitions of Rauzy classes. The first one was introduced by Rauzy and was in particular used by Veech to prove the ergodicity of the Teichmüller flow. The second one is more recent and uses a ``labeling'' of the underlying intervals, and was used in the proof of some recent major results about the Teichmüller flow.

The Rauzy diagrams obtained from the second definition are coverings of the initial ones. In this paper, we give a formula that gives the degree of this covering. 

This formula is related to moduli spaces of \emph{framed} translation surfaces, which corresponds to surfaces where we label horizontal separatrices on the surface. We compute the number of connected component of these natural coverings of the moduli spaces of translation surfaces.

Delecroix~\cite{Delecroix2010} proved a recursive formula for the cardinality of the (reduced) Rauzy classes. Therefore, we also obtain formula for labeled Rauzy classes.
\end{abstract}

\maketitle

\setcounter{tocdepth}{1}
\tableofcontents

%%%%%%%%%%%%%%%%%%%%%%%%%%%%%%%%%%%%%%%%%%
%%%%%%%%%%%%%%%%%%%%%%%%%%%%%%%%%%%%%%%%%%
\section{Introduction}

Rauzy induction was introduced in \cite{Rauzy} as a tool to study interval exchange maps. It is a renormalization process that associates to an interval exchange map, another one obtained as a first return map on a well chosen subinterval.
After the major works of Veech~\cite{Veech82} and Masur~\cite{Masur82}, the Rauzy induction became a powerful tool to study the Teichmüller geodesic flow. 

A slightly different tool was used by Kerckhoff \cite{Kerckhoff}, Bufetov \cite{Buf}, and Marmi-Moussa-Yoccoz~\cite{Marmi:Moussa:Yoccoz}. 

It is obtained after labeling the intervals, and keeping track of them during the renormalization process. This small change was a significative improvement, and was used in the recent past to prove other important results about the Teichmüller geodesic flow, for instance the simplicity of the Liapunov exponents (Avila-Viana \cite{Avila:Viana}), or the exponential decay of correlations (Avila-Gouezel-Yoccoz \cite{AGY}). 

An interval exchange map is naturally decomposed into a continuous and a combinatorial datum. A \emph{Rauzy} class is a minimal set of such combinatorial data invariant by the combinatorial Rauzy induction. We will speak of \emph{reduced} or \emph{labeled} Rauzy classes depending whether we use the definition of Rauzy, or the other one. 

Let $k_1,\ldots ,k_r$ be some pairewise distinct nonnegative integers and let $n_1,\ldots ,n_r$ be positive integers such that $\sum_{i=1}^{r}n_i k_i=2g-2$. The stratum of the moduli space of Abelian differentials 
whose corresponding surfaces have precisely $n_i$ singularities of degree $k_i$, for all $i\in \{1,\dots ,r\}$, is usually denoted as $\mathcal{H}(k_1^{n_1},\ldots ,k_r^{n_r})$. In this notation, we implicitely assume that $k_i\neq  k_j$ for all $i\neq j$. A singularity of degree zero is by convention a regular marked point on the surface.

The Veech construction naturally associates to a Rauzy class a connected component $\mathcal{C}$ of a stratum $\mathcal{H}(k_1^{n_1},\ldots ,k_r^{n_r})$ of the moduli space of Abelian differentials, and an integer $k\in \{k_1,\dots ,k_r\}$ that corresponds to degree of the singularity attached on the left in the Veech construction. In the next statement, the pair $(\mathcal{C},k)$ will be refered to the data associated to a Rauzy class by the Veech construction.

\begin{theorem}\label{MT}
Let $R_{lab}$ be a labeled Rauzy class and let $R$ be the corresponding reduced one. Let $(\mathcal{C},k)$ be data associated to $R$ by the Veech construction, where $\mathcal{C}$ is a connected component of a stratum $\mathcal{H}(k_1^{n_1},\ldots ,k_r^{n_r})$  of the moduli space of Abelian differentials, and $k\in \{k_1,\dots ,k_r\}$. 
Let $n$ be number of singularities of degree $k$ for a surface in $\mathcal{H}(k_1^{n_1},\dots ,k_r^{n_r})$.
We have:
$$
\frac{|R_{lab}|}{|R|}=\frac{\Pi_{i=1}^r n_i! (k_i+1)^{n_i}}{n(k+1)} \varepsilon
$$
where $\varepsilon$ satisfies: 
\begin{itemize}
\item $\varepsilon=\frac{1}{g}$  if $\mathcal{C}$ consists only of hyperelliptic surfaces with two conical singularities of degree $g-1$ and possibly some regular marked points.
\item $\varepsilon=\frac{1}{2}$ if $\mathcal{C}$ contains nonhyperelliptic surfaces and if there exists $k'\in \{k_1,\dots ,k_r\}$ which is odd.  
\item $\varepsilon=1$ in all the other cases.
\end{itemize}
\end{theorem}

Delecroix~\cite{Delecroix2010} has found a recursive formula for the cardinality of the reduced Rauzy classes. The previous theorem complete his result for the case of labeled Rauzy classes.

Our result is related to moduli spaces of framed translation surfaces. 
Informally, we will call a \emph{frame} on a translation surface $X$ a map $F_X$ from a discrete alphabet $\mathcal{A}$, to a set $\mathcal{S}_X$ of discrete combinatorial data on the surface $X$ such that  the moduli space of framed translation surfaces is a covering of the corresponding moduli space of translation surface.

In our case, we are interested in the case where $\mathcal{S}_X$ is the set of \emph{horizontal outgoing separatrices}.  Several different kind of frames will appear in this context. The most important will be the space $\mathcal{C}(\mathcal{F}_{comp})$. It corresponds to the space were we label exactly one horizontal separatrix for each singularity (see Section~\ref{framed} for a more precise definition), and the alphabet $\mathcal{A}$ is a disjoint union of subalphabets $\mathcal{A}_k$, such that the labels in $\mathcal{A}_k$ correspond to singularities of degree $k$.

Choosing $\alpha$ in some $\mathcal{A}_k$, one gets a natural covering $p_\alpha:\mathcal{C}(\mathcal{F}_{comp})\to \mathcal{C}(\mathcal{F}_{k})$, were $\mathcal{C}(\mathcal{F}_{k})$ is a moduli space of framed translation surfaces that corresponds to translation surfaces with a with a single marked separatrix adjacent to a singularity of degree $k$. 

Theorem~\ref{MT} will be a consequence of the two following results, which give a geometrical interpretation of the formula.

\begin{proposition}\label{prop:intro}
Let $R_{lab}$ be a labeled Rauzy class and let $R$ be the corresponding reduced one. The ratio $|R_{lab}|/|R|$ is equal to the degree of the covering  $p_\alpha$, when restricted to a connected component of $\mathcal{C}(\mathcal{F}_{comp})$.
\end{proposition}

\begin{theorem}\label{th:cc:Cfr}
Let $\mathcal{C}$ be a connected component of a stratum of the moduli space of Abelian differentials. The space $\mathcal{C}(\mathcal{F}_{comp})$  is not connected in general. More precisely, it has:
\begin{itemize}
\item $g$ connected components if $\mathcal{C}$ is the hyperelliptic connected component of $\mathcal{H}(g-1,g-1)$, with possibly some regular marked points.
\item Two connected components if it does not corresponds to a hyperelliptic connected component and if there exists odd degree conical singularities.  
\item One connected component in all the other cases.
\end{itemize}
 \end{theorem}

\subsection*{Acknowledgements}

We thank Vincent Delecroix, Luca Marchese and Erwan Lanneau for remarks and comments on this paper. We also thank~\cite{sage} for computational help.

%%%%%%%%%%%%%%%%%%%%%%%%%%%%%%%%%%%%%%%%%%
%%%%%%%%%%%%%%%%%%%%%%%%%%%%%%%%%%%%%%%%%%
\section{Background}
%==================================================
\subsection{Labeled and reduced interval exchange transformations}
We give here the two definitions of interval exchange transformations that are used in the literature. In order to distinguish them, we will add the terms \emph{reduced} and \emph{labeled}.

The first one is due to Rauzy \cite{Rauzy}
\begin{definition}
Let $I \subset \mathbb R$ be an open interval and let us choose a finite subset $\Sigma= \{x_1,\ldots,x_{d-1}\}$ of $I$. The complement of $\Sigma$ in $I$ is a disjoint union of $d\geq 2$ open subintervals $\{I_j, \ j=1,\dots,d \}$. An \emph{reduced interval exchange transformation} is a map $T$ from $I\backslash \Sigma$ to $I$ that permutes, by translation, the subintervals $I_j$. It is easy to see that $T(I\backslash \Sigma)=I\backslash \Sigma'$, were $\Sigma'\subset I$ is of the same cardinality as $\Sigma$, and that $T$ is precisely determined by:
\begin{itemize}
\item A \emph{combinatorial datum}: a permutation $\pi\in \Sigma_d$ which expresses that the interval number $k$, when counted from the left to the right, is sent to the place $\pi(k)$ by the map $T$.
\item A \emph{continuous datum}: a vector $\lambda\in \mathbb{R}^d$ with positive entries that corresponds to the lengths of the intervals.
\end{itemize}
\end{definition}
We will identify an interval exchange transformation with its parameters $(\pi,\lambda)$.
\medskip

The second definition of interval exchange transformation was first introduced by Kerckhoff~\cite{Kerckhoff} and later formalized by Bufetov~\cite{Buf} and  Marmi, Moussa \& Yoccoz~\cite{Marmi:Moussa:Yoccoz}. As we will see later, it simplifies the description of the Rauzy induction, but we get bigger Rauzy classes.

\begin{definition}
A \emph{labeled interval exchange map} is an reduced interval exchange map with a pair $(\pi_t,\pi_b)$ of one-to-one maps from a finite alphabet $\mathcal{A}$ to $\{1,\ldots,d\}$. The interval number $k$, when counted from the left to the right, is denoted by $I_{\pi_t^{-1}(k)}$. Once the intervals are exchanged, the interval number $k$ is $I_{\pi_b^{-1}(k)}$. 
\end{definition}
In the previous definition, it is easy to see that the permutation corresponding to the underlying reduced interval exchange map is $\pi=\pi_b \circ \pi_t^{-1}$, and the continuous datum is a vector with positive entries $\lambda\in \mathbb{R}^{\mathcal{A}}$.
We will identify a labeled interval exchange map with the pair $(\tilde{\pi},\lambda)$, were $\tilde{\pi}=(\pi_t,\pi_b)$. We will call $\tilde{\pi}$ a \emph{labeled permutation}.

We will usually represent a labeled permutation by a table:
\begin{eqnarray*}
\tilde{\pi}= \left(\begin{array}{ccccc}\pi_t^{-1}(1)&\pi_t^{-1}(2)&\ldots&\pi_t^{-1}(d)
  \\
\pi_b^{-1}(1)&\pi_b^{-1}(2)&\ldots&\pi_b^{-1}(d) \end{array}\right).
\end{eqnarray*}
As we can see from this representation, we have ``t'' for top, ``b'' for bottom in the notation  $\pi_t,\pi_b$.

A \emph{renumbering} of a labeled permutation is the composition of $(\pi_t,\pi_b)$ by a one-to-one map $f$ from $\mathcal{A}$ to $\mathcal{A}'$. It just corresponds to changing the labels without changing the underlying permutation. From the precious definitions, it is clear that a reduced interval exchange transformation (\emph{resp.} a permutation) is an equivalent class of labeled interval exchange maps (\emph{resp.} labeled permutations) up to renumbering. We will sometime identify a permutation with its unique representative $(\pi_t,\pi_b)$ with $\mathcal{A}=\{1,\ldots,d\}$ and $\pi_t=Id$.

%==================================================
\subsection{Rauzy-Veech induction}
Let $T$ be a labeled or reduced interval exchange map. 
The Rauzy--Veech induction $\mathcal R(T)$ of $T$ is defined as the first
return map of $T$ to a certain subinterval $J$ of $I$ (see~\cite{Rauzy,Marmi:Moussa:Yoccoz} and \cite{Rauzy} for details). 

We recall briefly the construction for the labeled case. Following the terminology of~\cite{Marmi:Moussa:Yoccoz} we define the \emph{type} of $T$ by $t$ if $\lambda_{\pi_t^{-1}(d)} > \lambda_{\pi_b^{-1}(d)}$ and $b$ if $\lambda_{\pi_t^{-1}(d)} < \lambda_{\pi_b^{-1}(d)}$. When $T$ is of type $t$ (respectively, $b$) we will say that the label $\pi_t^{-1}(d)$ (respectively, $\pi_b^{-1}(d)$) is the winner and that $\pi_b^{-1}(d)$ (respectively, $\pi_t^{-1}(d)$) is the looser. We define a subinterval $J$ of $I$ by
$$
J=\left\{ \begin{array}{ll}
I \backslash T(I_{\pi_b^{-1}(d)}) & \textrm{if $T$ is of type t};\\
I \backslash I_{\pi_t^{-1}(d)} & \textrm{if $T$ is of type b.} 
\end{array} \right.
$$
The image of $T$ by the  Rauzy-Veech induction $\mathcal R$ is defined as the first
return map of $T$ to the subinterval $J$. This is again a labeled interval exchange transformation, defined on $d$ letters. The combinatorial datum of this new interval exchange transformation is very easy to calculate in terms of the one of $T$. Indeed, let 
$\alpha\in \mathcal{A}$ (\emph{resp.} $\beta\in \mathcal{A}$) be the winner (\emph{resp.} looser). Let $\lambda' \in \mathbb{R}^\mathcal{A}$ such that:
\begin{eqnarray*}
\lambda'_\alpha&=&\lambda_\alpha-\lambda_\beta \\
\lambda'_\nu&=&\lambda_\nu \qquad \textrm{for all } \nu \in  \mathcal{A}\backslash\{\alpha\}
\end{eqnarray*}

Then, $\mathcal{R}(T)=(\mathcal{R}_\varepsilon(\pi),\lambda')$, were $\varepsilon$ is the type of $T$, and $\mathcal{R}_t,\mathcal{R}_b$ are the following combinatorial maps:

\begin{enumerate}
\item {$\mathcal{R}_t$}: let $k=\pi_b(\pi_t^{-1}(d))$ with $k\leq d-1$. 
Then, $\mathcal R_t(\pi_t,\pi_b)=(\pi_t',\pi_b')$ where $\pi_t=\pi_t'$ and
$$
\pi_b'^{-1}(j) = \left\{
\begin{array}{ll}
\pi_b^{-1}(j) & \textrm{if $j\leq k$}\\
\pi_b^{-1}(d) & \textrm{if $j = k+1$}\\
\pi_b^{-1}(j-1) & \textrm{otherwise.}
\end{array} \right.
$$

\item {$\mathcal{R}_b$}: let $k=\pi_t(\pi_b^{-1}(d))$ with $k \leq
d-1$. Then, $\mathcal R_b(\pi_t,\pi_b)=(\pi_t',\pi_b')$ where $\pi_b=\pi_b'$ and
$$
\pi_t'^{-1}(j) = \left\{
\begin{array}{ll}
\pi_t^{-1}(j) & \textrm{if $j\leq k$}\\
\pi_t^{-1}(d) & \textrm{if $j = k+1$}\\
\pi_t^{-1}(j-1) & \textrm{otherwise.}
\end{array} \right.
$$
\end{enumerate}

The maps $\mathcal{R}_t,\mathcal{R}_b$ are called the combinatorial Rauzy moves. It is easy to define similar maps for reduced permutations. We first identify $\pi$ with the corresponding $(Id,\pi_b)$ and perform the Rauzy move. Then, if needed, we renumber the result so that it corresponds to a reduced permutation. 
We will still denote the corresponding maps by $\mathcal{R}_t,\mathcal{R}_b$, 
since it will always be clear, when the distinction is needed, whether or not the objects are labeled or reduced.

We define the Rauzy induction for reduced interval exchange maps by considering labeled interval exchange maps up to renumbering.

\begin{definition}
A \emph{Rauzy class}, usually denoted by $R$ is a minimal  set of labeled or reduced  permutations invariant by the combinatorial  Rauzy moves. 

A \emph{Rauzy diagram} is a graph whose vertices are the elements of a Rauzy class and whose vertices are the combinatorial Rauzy maps.

A Rauzy class or Rauzy diagram will be called labeled or reduced depending on the corresponding permutations.
\end{definition}

\begin{example}\label{ex:card:rc}
\begin{itemize}
\item Let $\tau_n=\left(\begin{smallmatrix}
1&2&\ldots &n\\ n&n-1&\ldots &1
\end{smallmatrix}\right)$. Rauzy proved that the cardinality of $R(\tau_n)$ is $2^{n-1}-1$ for the reduced case, and one can prove ``by hand'' that it is the same for the labeled case.     
\item Let $\pi_n= \left( \begin{smallmatrix} 0 & 2 & 3 & \dots & n-1 & 1 & n \\
n & n-1 & \dots & 3 & 2 & 1 & 0 \end{smallmatrix} \right)$. 
Contrary to the previous case, the labeled and reduced diagrams are not 
isomorphic anymore. The structure of the labeled and reduced Rauzy diagrams is  precisely described in \cite{BL:pA}. It is in particular shown that the cardinality of the reduced diagram is $2^{n-1} -1 + n$ and 
the cardinality of the labeled diagram is $(2^{n-1} -1+ n)(n-1)$.
 
\item Consider $\pi=\left( \begin{smallmatrix}
1&2&3&4&5&6&7&8&9\\ 9&1&4&3&2&5&8&7&6
\end{smallmatrix} \right)$. The reduced Rauzy class is of size\footnote{This can be computed for instance using Zorich's MATHEMATICA  software, or using the SAGE package developed by Delecroix} 1255 and the labeled Rauzy class is of size 30120. The ratio is 24.  
\end{itemize}
\end{example}

%==================================================
\subsection{Translation surfaces and moduli space}
%
%---------------------------------------------------
\subsubsection{Translation surfaces}
A \emph{translation surface} is a (real, compact, connected) genus $g$ surface $X $ with a translation atlas \emph{i.e.} a triple $(X,\mathcal U,\Sigma)$ such that $\Sigma$ is a finite subset of $X$ (whose elements are called {\em singularities}) and $\mathcal U = \{(U_i,z_i)\}$ is an atlas of $X \setminus \Sigma$ whose transition maps are translations. We will require that  for each $s\in \Sigma$, there is a neighborhood of $s$ isometric to a Euclidean cone.  One can show that the holomorphic structure on $X\setminus \Sigma$ extends to $X$ and that the holomorphic 1-form $\omega=dz_i$ extends to a holomorphic $1-$form on $X$ were  $\Sigma$ corresponds to the zeroes of $\omega$ and maybe some marked points. We usually call $\omega$ an \emph{Abelian differential}. 

For $g \geq 1$, we define the moduli space of Abelian differentials $\HH_g$ as the moduli space of pairs $(X,\omega)$ where $X$ is a genus $g$ (compact, connected) Riemann surface and $\omega$ non-zero holomorphic $1-$form defined on $X$. The term moduli space means that we identify the points $(X,\omega)$ and $(X',\omega')$ if there exists an analytic isomorphism $f:X \rightarrow X'$ such that 
$f^* \omega'=\omega$.
The group $\textrm{SL}_2(\mathbb R)$ naturally acts on the moduli space of translation surfaces by post composition on the charts defining the translation structures.

One can also see a translation surface obtained as a polygon (or a finite union of polygons) whose sides come by pairs, and for each pairs, the corresponding segments are parallel and of the same lengths. These parallel sides are glued together by translation and we assume that this identification preserves the natural orientation of the polygons. In this context, two translation surfaces are identified in the moduli space of Abelian differentials if and only if the corresponding polygons can be obtained from each other by cutting and gluing and preserving the identifications. Also, the $SL_2(\mathbb{R})$ action in this representation is just the natural linear action on the polygons. \medskip	

The moduli space of Abelian differentials is stratified by the  combinatorics of the zeroes; we will denote by $\HH(k_1^{n_1},\ldots ,k_r^{n_r})$ the stratum of $\HH_g$ consisting of (classes of) pairs $(X,\omega)$ such that $\omega$ possesses exactly $n_i$ zeroes on $X$ with multiplicities $k_i$ for all $i\in \{1,\ldots ,r\}$, and no other zeroes. It is a well known part of the Teichmüller theory that these spaces are  (Hausdorff) complex analytic, and in fact algebraic, spaces. These strata are non-connected in general but each stratum has at most three 
connected components (see~\cite{KoZo} for a complete classification, or see Section~\ref{sec4}). 

Given a translation surface $X$, we will call \emph{separatrix} an oriented half line (possibly finite) starting from a singularity of $X$. A horizontal separatrix 
$l$ will be \emph{outgoing} if it goes on the right in a translation chart, and \emph{incoming} otherwise.

%====================================================
\subsection{Suspension data}
\label{sec:suspension}

The next construction provides a link between interval exchange transformations and translation surfaces.
A suspension datum for $T=(\pi,\lambda)$ is a collection of vectors $\{\tau_\alpha\}_{\alpha\in \mathcal{A}}$ such that
\begin{itemize}
\item $\forall 1 \leq k \leq d-1,\ \sum_{\pi_t(\alpha)\leq k} \tau_\alpha>0$,
\item $\forall 1 \leq k \leq d-1,\ \sum_{\pi_b(\alpha)\leq k} \tau_\alpha<0$. 
\end{itemize}

We will often use the notation $\zeta=(\lambda,\tau)$. 
To each suspension datum $\tau$, we can associate a translation surface $(X,\omega)=X(\pi,\zeta)$ 
in the following way. 

Consider the broken line $L_t$ on $\mathbb{C}=\mathbb R^2$ defined by concatenation of the vectors $\zeta_{\pi_t^{-1}(j)}$ (in this order) for $j=1,\dots,d$ with starting point at the origin. Similarly, we consider the broken line $L_b$  defined by concatenation of the vectors $\zeta_{\pi_b^{-1}(j)}$ (in this order) for $j=1,\dots,d$ with starting point at the origin. If the lines $L_t$ and $L_b$ have no intersections other than the endpoints, we can construct a translation surface $X$ by identifying each side $\zeta_j$ on $L_t$ with the side $\zeta_j$ on $L_b$ by a translation. The resulting surface is a translation surface endowed with  the form $\omega=dz$. Note that the lines $L_t$ and $L_b$ might have some other intersection points. But in this case, one can still define a translation surface by using the \emph{zippered rectangle construction}, due to  Veech (\cite{Veech82}). See for instance Figure~\ref{ietsusp}.

\begin{figure}[htb]
\begin{center}
\includegraphics[scale=0.7]{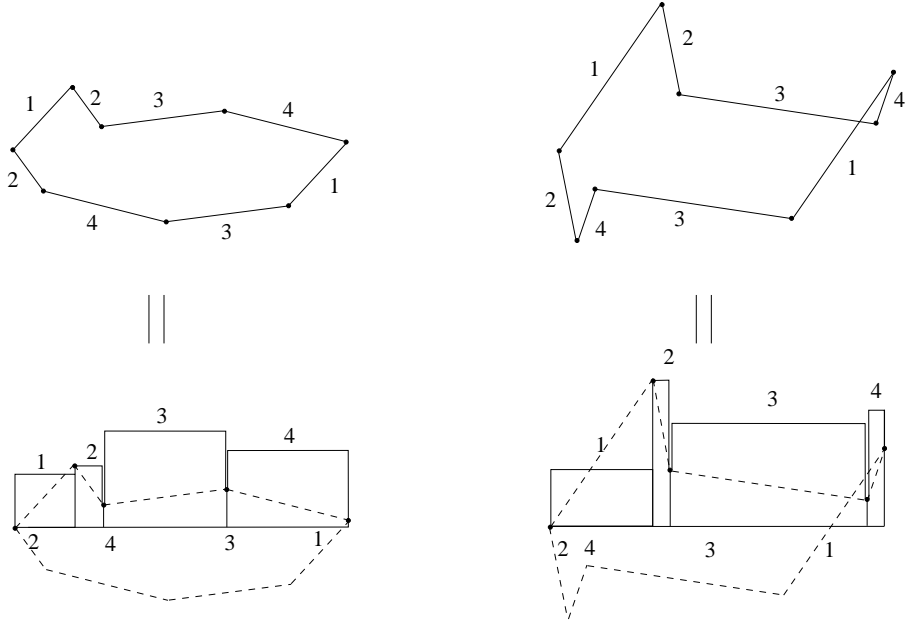}
\caption{The zippered rectangle construction, for two examples of suspension data.}
\label{ietsusp}
\end{center}
\end{figure}

Let $I \subset X$ be the horizontal interval defined by $I = (0,\sum_{\alpha} \lambda_{\alpha}) \times \{0\}$. The reduced interval exchange transformation $T$ is precisely the one defined by the first return map to $I$ of the vertical flow on $X$. 

We can extend the Rauzy induction to suspension data in the following way:  let $\tau$ be a suspension data over $(\pi,\lambda)$,  we define $\mathcal R(\pi,\lambda,\tau)=(\pi',\lambda',\tau')$ by:
\begin{itemize}
\item $\mathcal{R}(\pi,\lambda)=(\pi',\lambda')$
\item $\tau'_\alpha=\tau_\alpha-\tau_\beta$, where $\alpha$ (resp. $\beta$) is the winner (resp. looser) for $T=(\pi,\lambda)$
\end{itemize}

This extension is known as the \emph{Rauzy--Veech induction}, and is used as a discretization of the Teichmüller flow.

\begin{remark}
\label{rk:isometric}
By construction the two translation surfaces $X(\pi,\zeta)$ and $X(\pi',\zeta')$ define the same element in the moduli space.
\end{remark}

\begin{remark}
Note that $\lambda,\tau$ define natural local parameters for the stratum of the moduli space of Abelian differentials.
\end{remark}

\begin{remark}\label{data:separatrix}
The left end of the two lines in the previous construction $L_t,L_b$ is a singularity, and the horizontal half line starting from this point to the right corresponds to a choice of a horizontal separatrix starting from this singularity, and it is easy to see that this combinatorial data is preserved under the Rauzy--Veech induction. Let us denote by $l(\pi,\zeta)$ this separatrix.
\end{remark}

%%%%%%%%%%%%%%%%%%%%%%%%%%%%%%%%%%%%%%%%%%
%%%%%%%%%%%%%%%%%%%%%%%%%%%%%%%%%%%%%%%%%%
\section{Rauzy diagrams and framed translation surfaces}
%==================================================
\subsection{Moduli space of framed translation surface.}\label{framed}

A \emph{frame} on a translation surface $X$ a map $F_X$ from a discrete alphabet $\mathcal{A}$, to a set $\mathcal{DC}_X$ of discrete combinatorial data on the surface $X$. 

Let  $\mathcal{C}$ be a connected component of a stratum of moduli space of translation surfaces, and a collection $\mathcal{F}$ of frames for translation surfaces in $\mathcal{C}$, with a fixed alphabet.  One can define the corresponding moduli space of framed translation surfaces: two elements $(X,F_X)$ and $(X',F'_{X'})$ are identified if there is a translation mapping $X\to X'$ which is consistent with the frames. Then, we will denote by $\mathcal{C}(\mathcal{F})$ the corresponding moduli space.

The sets $\mathcal{DC}_X$ can be many things: incoming or outgoing horizontal separatrices, $H_1(X,\mathbb{Z})$, etc\dots Here we will not study the precise conditions on the collections of frames so that $\mathcal{C}(\mathcal{F})$ is a ``nice'' space. We will just introduce three cases, for which $\mathcal{C}(\mathcal{F})$ are coverings of $\mathcal{C}$.
We first ask that $\mathcal{DC}_X=\mathcal{S}_X$ is the set of \emph{horizontal outgoing separatrices}, then we consider the three families of all frames that satisfy the following conditions respectively:

\begin{enumerate}
\item $\mathcal{F}_k$: the set $\mathcal{A}$ is a singleton and the image of $F_X$  is any separatrix adjacent to a degree $k$ singularity. 

\item
$\mathcal{F}_{sat}$:  $F_X$ is a one-to-one mapping from $\mathcal{A}$ to $\mathcal{S}_X$. 

\item $\mathcal{F}_{comp}$: $\mathcal{A}=\sqcup_k \mathcal{A}_k$, the map $F_X$ is injective, and we require that for each $k$ any singularity of $S$ of degree $k$ has a unique separatrix in $F(\mathcal{A}_k)$.
\end{enumerate}

We will denote respectively by $\mathcal{C}(\mathcal{F}_{k})$, $\mathcal{C}(\mathcal{F}_{sat})$, and $\mathcal{C}(\mathcal{F}_{comp})$, the corresponding moduli spaces, which are finite coverings of $\mathcal{C}$.

\begin{enumerate}
\item The space $\mathcal{C}(\mathcal{F}_{k})$ is the moduli space of pairs $(X,l)$, were $X\in \mathcal{C}$ and $l$ is a separatrix adjacent to a singularity of degree $k$ in $X$.
\item The space $\mathcal{C}(\mathcal{F}_{sat})$ corresponds to translation surfaces were we label each horizontal outgoing separatrix by an element in $\mathcal{A}$. It corresponds to a ``saturated case''.
\item The space  $\mathcal{C}(\mathcal{F}_{comp})$ corresponds to translation surfaces were we label exactly one separatrix for each singularity, accordingly to the degree of the singularity. 
\end{enumerate}

As we will see, the spaces  $\mathcal{C}(\mathcal{F}_{k})$ (resp.   $\mathcal{C}(\mathcal{F}_{sat})$) will appear naturally in the study of reduced (resp. labeled) Rauzy classes. We will then reduce the problem to the study of the space $\mathcal{C}(\mathcal{F}_{comp})$.

Note that the space $\mathcal{C}(\mathcal{F}_{comp})$ was also introduced recently by Marchese in \cite{Marchese} (Section~3.2).

\subsection{Moduli space of reduced suspension data.}
The set of suspension data associated to a labeled or reduced permutation is connected (in fact, convex).  Hence, for a Rauzy class $R$, all flat surfaces obtained from the Veech construction are in the same connected component of a stratum of the moduli space of Abelian differentials. 
In fact, according to Remark~\ref{data:separatrix} there is a natural map:
\begin{eqnarray*}
\Phi:\mathcal{H}_{R}=\left\{
\begin{array}{l}
(\pi,\zeta), \pi\in R,\\ \zeta \textrm{ susp. dat. for }\pi
\end{array}
\right \} {/ \mathcal{R}}&\rightarrow&  
\mathcal{C}(\mathcal{F}_{k}) \\
\left[(\pi,\zeta) \right] \qquad
 &\mapsto &
\big(X(\pi,\zeta),l(\pi,\zeta) \big)
\end{eqnarray*}

\begin{example}
Let us consider the permutations given in Example~\ref{ex:card:rc}. We have:
\begin{itemize}
\item For $\tau_n=\left(\begin{smallmatrix}
1&2&\ldots &n\\ n&n-1&\ldots &1
\end{smallmatrix}\right)$, the corresponding connected component is $\mathcal{H}^{hyp}(n-2)$ or $\mathcal{H}^{hyp}(\frac{n-1}{2}-1,\ \frac{n-1}{2}-1)$ depending on the parity of $n$. 
\item For $\pi_n= \left( \begin{smallmatrix} 0 & 2 & 3 & \dots & n-1 & 1 & n \\
n & n-1 & \dots & 3 & 2 & 1 & 0 \end{smallmatrix} \right)$, the corresponding connected component is $\mathcal{H}^{hyp}(0, n-2)$ or $\mathcal{H}^{hyp}(0, \frac{n-1}{2}-1,\frac{n-1}{2}-1)$, the singularity which is marked by the Veech construction being of degree 0.
 
 \item For $\pi=\left( \begin{smallmatrix}
1&2&3&4&5&6&7&8&9\\ 9&1&4&3&2&5&8&7&6
\end{smallmatrix} \right)$, the corresponding stratum is  $\mathcal{H}(1,1,1,1)=\mathcal{H}(1^4)$ and is connected.
\end{itemize}

\end{example}

When $R$ is a Rauzy class of reduced permutations, then the following theorem was proven in \cite{Boissy:rc}.

\begin{theorem}\label{th:Boissy:rc}
The map $\Phi$ is a homeomorphism on its image. The complement of the image of $\Phi$ is contained in a codimension~2 subset of $\mathcal{C}(\mathcal{F}_{k})$, which is connected. \end{theorem}

%==================================================
\subsection{Moduli space of labeled suspension data}

The previous section describes what represents ``geometrically'' a reduced Rauzy diagram (or more precisely, the corresponding moduli space of suspension data): a suspension data for a reduced permutation, modulo the Rauzy--Veech induction corresponds to a translation surface with a marked separatrix, \emph{i.e.} an element of $\mathcal{F}_k(\mathcal{C})$.

In this section, we give an analogous description of a labeled Rauzy diagram. In this case, a suspension data for a labeled permutation, modulo the Rauzy--Veech induction corresponds to a translation surface where all the  separatrices are labeled, \emph{i.e} an element of $\mathcal{C}(\mathcal{F}_{sat})$. Then, we will show next  that we can reduce the problem to studying the moduli space of translation surfaces with a single marked separatrix for each singularities.

\begin{figure}[htb]
\begin{center}
\labellist
\tiny 
\hair 2pt
\pinlabel $l_1$ at 15 45
\pinlabel $l_2$ at 40 78
\pinlabel $l_3$ at 55 58
\pinlabel $l_4$ at 125 65
\normalsize
\pinlabel 1 at 16 62 
\pinlabel 2 at 40 58
\pinlabel 3 at 80 60
\pinlabel 4 at 150 58
\pinlabel 2 at 8 25 
\pinlabel 4 at 50 5
\pinlabel 3 at 120 0
\pinlabel 1 at 165 20
\endlabellist

\includegraphics[scale=1.4]{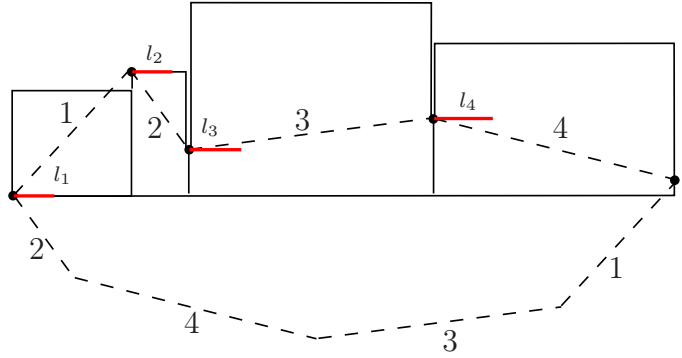}
\caption{A framing of a surface issued from the Veech construction. Here we have $l_i=F(i)$ for  $i=1,2,3,4$ and $l_1=l_2$.}
\label{complab}
\end{center}
\end{figure}

Let $\pi=(\pi_t,\pi_b)$  be a labeled permutation and let $\zeta$ be a suspension data for $\pi$. The zippered rectangle construction naturally defines a framed translation surface (see Figure~\ref{complab}) in the following way: for each rectangle in the Veech construction, the left vertical side contains a unique singularity. Hence, we can label the corresponding outgoing horizontal separatrix with the letter of the rectangle. Furthermore, two of these rectangles (corresponding to $\pi_t^{-1}(1)$ and $\pi_b^{-1}(1)$) intersect the corresponding singularity at a left corner (the bottom for one, the top for the other), and the corresponding horizontal outgoing separatrix is the same, so is labeled twice: once by the symbol $\pi_{t}^{-1}(1)$, and once by the symbol $\pi_b^{-1}(1)$. For all the other rectangles, the singularity on the left is in the interior of the left vertical side, hence, each corresponding separatrix is uniquely labeled. Therefore one gets a one-to-one map:
$$
F:\mathcal{A}' \to \mathcal{S}_X
$$
were $\mathcal{A}'$ is the quotient of $\mathcal{A}$ by the equivalence relation $\pi_t^{-1}(1)\sim \pi_b^{-1}(1)$. 

The element $F(\{\pi_t^{-1}(1), \pi_b^{-1}(1)\})$ will be refered as the \emph{doubly labeled} separatrix.

\begin{lemma}\label{Rlab:to:Hsat}
Let $(\pi,\zeta)$ as previously. The framed translation surface constructed as before from $(\pi,\zeta)$ and the one defined by $\mathcal{R}(\pi,\zeta)$ are the same element in $\mathcal{C}(\mathcal{F}_{sat})$. 
\end{lemma} 

\begin{proof}
This is an elementary check. 
\end{proof}

The following proposition transforms the initial combinatorial question into a topological one on the moduli space of Abelian differentials.

\begin{proposition}\label{rc:to:cc}
There is a natural  one-to-one correspondence between labeled Rauzy classes and connected components of the moduli space of framed translation surfaces $\mathcal{C}(\mathcal{F}_{sat})$. The degree of the mapping from a labeled Rauzy diagram to the reduced one is then precisely the degree of the natural mapping from a connected component of $\mathcal{C}(\mathcal{F}_{sat})$ to $\mathcal{C}(\mathcal{F}_{k})$.
\end{proposition}

\begin{proof}
Let $R_{all}$ be the set of labeled permutations that corresponds to $\mathcal{C}(\mathcal{F}_{k})$, and let: 
$$\mathcal{H}_{R_{all}}=\{(\pi,\zeta),\pi\in R_{all}, \ \zeta \textrm{ suspension data for }\pi\} /\mathcal{R}.$$
By Lemma~\ref{Rlab:to:Hsat} there is a map $\Phi_{sat}: \mathcal{H}_{R_{all}}\to\mathcal{C}(\mathcal{F}_{sat})$ such that the following diagram commutes.
$$
\begin{array}{ccc}
\mathcal{H}_{R_{all}} &\xrightarrow{\Phi_{sat}} & \mathcal{C}(\mathcal{F}_{sat})\\
\downarrow {p_0}& & \downarrow {p_1} \\
\mathcal{H}_{R}& \xrightarrow{\ \Phi\ } &\mathcal{C}(\mathcal{F}_{k})
\end{array}
$$
Where $p_0$ is the canonical map that replace a labeled permutation by a reduced one, and $p_1$ is the map that ``forget'' all labels except for the doubly labeled separatrix.
The maps $\Phi$ and $\Phi_{sat}$ are homeomorphisms on their images, and onto up to codimension 2 subsets (see \cite{Boissy:rc}, Section 2 for details.). Hence, $\mathcal{H}_{R_{all}}$ and $\mathcal{C}(\mathcal{F}_{sat})$ have the same number of connected components, and the degree of the maps $p_1$ and $p_{0}$, when restricted to a connected component are the same. But the degree of the map $p_{0}$ restricted to a connected component is precisely the degree of natural map from the labeled Rauzy diagram to the reduced one.
\end{proof}

%==================================================
\subsection{Moduli space of translation surfaces with frame.} $\ $\\

There are obvious invariants for the connected components of the moduli space $\mathcal{C}(\mathcal{F}_{sat})$. Indeed, two elements of $\mathcal{C}(\mathcal{F}_{sat})$ that are in the same connected component must satisfy the following property:
\begin{itemize}
\item The labels that correspond to a given singularity on one surface must correspond to a same singularity on the other surface.
\item The canonical cyclic order on the set of labels obtained by rotating clockwise around a singularity must be the same.
\end{itemize}

Hence, a connected component of $\mathcal{C}(\mathcal{F}_{sat})$ is clearly isomorphic to a connected component of $\mathcal{C}(\mathcal{F}_{comp})$

Let $\alpha\in \mathcal{A}_k$ be a label associated to a degree $k$ singularity. There is a natural covering $p_{\alpha}$ from $\mathcal{C}(\mathcal{F}_{comp})$ to $\mathcal{C}(\mathcal{F}_{k})$ obtained by ``forgetting'' all the markings, except the one that corresponds to $\alpha$. 
The following proposition summarizes the discussion of this section, and is equivalent to Proposition~\ref{prop:intro}.

\begin{proposition}\label{deg:coverings}
Let $R_{lab}$ be a labeled Rauzy class and $R$ be the corresponding reduced one. Let $k$ be the degree of the marked singularity associated to $R$. 
The ratio $\frac{|{R_{lab}}|}{|R |}$ equals the degree of the canonical projection $p_\alpha:\mathcal{C}(\mathcal{F}_{comp})\to \mathcal{C}(\mathcal{F}_{k})$, restricted to a connected component of $\mathcal{C}(\mathcal{F}_{comp})$, where $\alpha$ is a label associated to a degree $k$ singularity.
\end{proposition}

\begin{proof}
It was proven in \cite{Boissy:rc} that $\mathcal{C}(\mathcal{F}_{k})$ is connected.   
A connected component of $\mathcal{C}(\mathcal{F}_{comp})$ is naturally isomorphic to a connected component of $\mathcal{C}(\mathcal{F}_{sat})$. Then, we just apply Proposition~\ref{rc:to:cc}.
\end{proof}

%%%%%%%%%%%%%%%%%%%%%%%%%%%%%%%%%%%%
%%%%%%%%%%%%%%%%%%%%%%%%%%%%%%%%%%%%
\section{Topological invariants for framed translation surfaces}\label{sec4}

From now, a framed translation surface will be an element in $\mathcal{C}(\mathcal{F}_{comp})$.

As seen in Proposition~\ref{deg:coverings}, the formula given in Theorem~\ref{MT} is related to the number of connected components of  $\mathcal{C}(\mathcal{F}_{comp})$. 
Also, the degree of the covering $\mathcal{C}(\mathcal{F}_{comp})\to \mathcal{C}$, restricted to a connected component of $\mathcal{C}(\mathcal{F}_{comp})$, is clearly $\frac{\Pi_{i=1}^r n_i! (k_i+1)^{n_i}}{c}$, were $c$ is the number of connected component of $\mathcal{C}(\mathcal{F}_{comp})$, since ${\Pi_{i=1}^r n_i! (k_i+1)^{n_i}}$ is the number of possible frame $F\in \mathcal{F}_{comp}$ on a surface.

In this section, we give lower bounds on the number of connected components of  $\mathcal{C}(\mathcal{F}_{comp})$. There are two cases.
\begin{itemize}
\item The ``hyperelliptic case''. If the corresponding surfaces are all hyperelliptic and have two singularities of degree $g-1$, with possibly some added regular marked points. Then, $\mathcal{C}(\mathcal{F}_{comp})$ cannot be connected due to the extra symmetries of the underlying translation surfaces.
\item The ''odd singularity case''. When there are odd degree singularities, we can define on $\mathcal{C}(\mathcal{F}_{comp})$ a topological invariant which generalizes the well known \emph{spin structure invariant} for the moduli space of Abelian differentials, found by Kontsevich and Zorich. 
\end{itemize}

Recall that a Riemann $S$ surface is hyperelliptic if there exists an involution $\tau$ such that $S/\tau=\mathbb{CP}^1$. Since $\mathbb{CP}^1$ does not have any nontrivial Abelian differential, then for any translation surface $(S,\omega)$ such that $S$ is hyperelliptic the corresponding involution $\tau$ satisfies
$\tau^* \omega=-\omega$. In particular, this means that the translation surface $(S,\omega)$ have an isometric involution which reverse the vertical direction. Kontsevich and Zorich have shown that for each genus $g\geq 2$, there are exactly two strata that contain a connected component which consists only of hyperelliptic translation surfaces. These are the strata $\mathcal{H}(2g-2)$ and $\mathcal{H}(g-1,g-1)$. Of course, for each strata, one can also define new ones by adding regular marked points on the surfaces. 

\begin{proposition}\label{low:bound:hyp}
Assume that $\mathcal{C}$ consists only of hyperelliptic translation surfaces with two singularities of degree $g-1$ and $n_0$ regular marked points.  Then, $\mathcal{C}(\mathcal{F}_{comp})$ has at least $g$ connected components.  
\end{proposition}

\begin{proof}
Let $X\in \mathcal{C}(\mathcal{F}_{comp})$. 
We denote by $l_0$ and $l_1$ the marked separatrices associated to the degree $g-1$ singularities, and we denote by $P_i$ the singularity corresponding to $l_i$.

 The hyperelliptic involution interchanges $P_0$ and $P_1$. 
 Hence, there is a well defined (incoming) separatrix $l_1'$  adjacent to $P_1$ which is the image of $l_0$.  The angle $\theta$  between $l_1'$ and $l_1$ is an odd multiple of $\pi$ and is constant under continuous deformations of $X$ inside the ambient stratum. Note that the value of $\theta$ does not depend on any choice.
Hence the value of $\theta$ is an invariant of the  connected component of $\mathcal{C}(\mathcal{F}_{comp})$.  Since all the values $\pi,3\pi,\dots ,(2g-1)\pi$ are possible, we see that the number of connected components of $\mathcal{C}(\mathcal{F}_{comp})$ is at least $g$. 
\end{proof}

\begin{proposition}\label{prop:odd:sing}
Assume that $\mathcal{C}$ consists of translation surfaces with at least one odd degree singularity.  Then, $\mathcal{C}(\mathcal{F}_{comp})$ has at least $2$ connected components.  
\end{proposition}

We postpone the proof of this proposition to the end of this section. We first define the ``spin structure'' invariant for $\mathcal{C}(\mathcal{F}_{comp})$. 

Let $X$ be a completely framed surface with at least one odd degree singularity. Note that the number of odd degree singularities of $X$ is necessarily even, since the sum of the degree of the singularities must be equal to $2g-2$ by the Riemann-Roch formula.
For each singularity, we have given a name $\alpha\in \mathcal{A}$ to a horizontal outgoing separatrix. Now let us fix a total order on the finite alphabet $\mathcal{A}$, so that the marked separatrices are naturally ordered. This order induces an oriented pairing of the separatrices corresponding to odd degree singularities. 

\begin{figure}[htb]
\begin{center}
\labellist
\tiny 
\hair 2pt
\pinlabel $l_1^+$ at 120 400   \pinlabel $l_2^-$ at 80 470   \pinlabel $a$ at 50 490  \pinlabel $b$ at 120 510
\pinlabel $c$ at 160 530     \pinlabel $d$ at 220 530    \pinlabel $e$ at 260 500  \pinlabel $e$ at 50 410
\pinlabel $d$ at 100 370   \pinlabel $c$ at 160 385   \pinlabel $b$ at 200 400   \pinlabel $a$ at 280 420

\pinlabel $e$ at 40 60   \pinlabel $d$ at 80 20   \pinlabel $0$ at 130 -10  \pinlabel $c$ at 170 20
\pinlabel $b$ at 200 50     \pinlabel $a$ at 270 60   \pinlabel $a$ at 35 140   \pinlabel $0$ at 80 180  
 \pinlabel $b$ at 110 160   \pinlabel $c$ at 145 180   \pinlabel $d$ at 220 190 \pinlabel $e$ at 260 150

\pinlabel $e$ at 440 60 \pinlabel $d$ at 480 20   \pinlabel $0$ at 530 -10  \pinlabel $c$ at 570 20
\pinlabel $b$ at 600 50     \pinlabel $a$ at 670 60   \pinlabel $a$ at 435 140   \pinlabel $0$ at 480 180  
 \pinlabel $b$ at 510 160   \pinlabel $c$ at 545 180   \pinlabel $d$ at 620 190 \pinlabel $e$ at 660 150

\endlabellist
\includegraphics[scale=0.5]{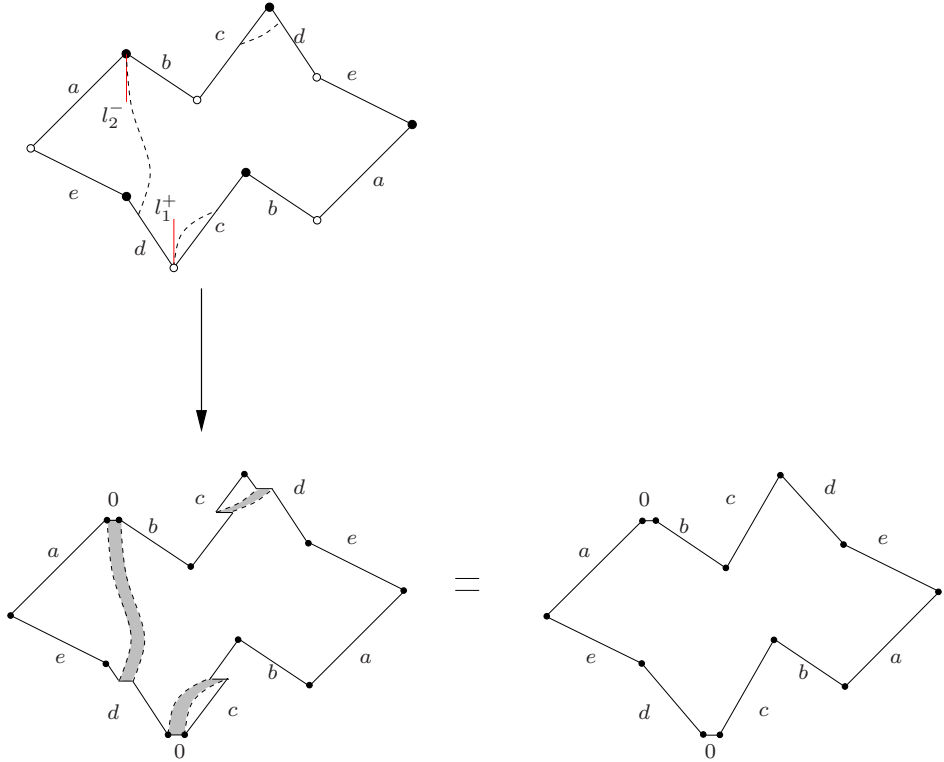}
\caption{Building a surface with only even degree singularities.}
\label{curvpoly}
\end{center}
\end{figure}

Now let $(l_1,l_2)$ be such a pair. We rotate the first separatrix clockwise by an angle $\pi/2$, and the second one counterclockwise by an angle $\pi/2$. We obtain pairs of vertical separatrices, the first one being on the positive direction, the second one on the negative direction. We denote by $(l_1^+,l_2^-)$ this pair of positive/negative vertical separatrices,  let us also denote by $k_1,k_2$ the degree of the corresponding zeroes. According to Hubbard--Masur \cite{Hubbard:Masur}, there exists a (smooth) path $\nu$ transverse to the horizontal foliation which starts being tangent to $l_1^+$ and ends being tangent to $l_2^-$. Now we consider the following surgery: we cut the surface $X$  along the path and paste in a ``curvilinear parallelogram'' with two small horizontal sides and two opposite sides that are isomorphic to $\nu$ (see Figure~\ref{curvpoly}).  Then, gluing together the horizontal sides of the parallelogram, one obtains a translation surface where the pair of singularities corresponding to $l_1^+,l_2^-$ have become a singularity of degree $k_1+k_2+2$, which is even. We will refer to this construction as the \emph{parallelogram construction with parameters} $(l_1,l_2)$.

Then, we apply this procedure on all the pairs of vertical separatrices that were defined previously. The resulting translation surface only has even singularities and is of genus at least 3, since the minimal genus case corresponds  to starting from $\mathcal{H}(1,1)$ and ending in $\mathcal{H}(1+1+2)$.

Recall that the strata of the moduli space of Abelian differentials corresponding to only even degree singularities are not connected as soon as the genus is greater than or equal to 3, and are distinguished by the parity of spin structure. We will prove the following lemma:
\begin{lemma}
The connected component of the resulting surface in the previous construction doesn't depend on the chosen paths. 
\end{lemma}

\begin{proof}
Up to a small deformation of the surface $X$, one can assume that it is obtained by the Veech construction starting from a data $(\pi,\lambda,\tau)$. Then, for a pair $l_1,l_2$ of separatrices as previously, the  surface obtained after the parallelogram construction with parameters $l_1,l_2$  also arises from the Veech construction, where the corresponding permutation is obtained from $\pi$ by  adding a new label on the top before the symbol corresponding to $l_2$ and the same label on the bottom before the label corresponding to $l_1$. For instance, in Figure~\ref{curvpoly}, the labeled permutation $\left(\begin{smallmatrix} a&b&c&d&e \\ e&d&c&b&a
\end{smallmatrix}\right)$ becomes $\left(\begin{smallmatrix} a&0&b&c&d&e \\ e&d&0&c&b&a
\end{smallmatrix}\right)$, since $l_1$ corresponds to $c$ and $l_2$ corresponds to $b$.

In particular, the permutation after removing all the odd degree singularities doesn't depend on the choices of the paths, but only on the order that we have chosen on $\mathcal{A}$. 
\end{proof}

\begin{proof}[Proof of Proposition~\ref{prop:odd:sing}]
We just need to show that the spin structure invariant defined before reaches all the possible values. First we recall Kontsevich--Zorich formula for  the parity of spin structure for a translation surface $X'$ of genus $g'$ with only even degree singularities. Let $a_1,b_1,\dots ,a_{g'},b_{g'}$ be a collection of closed paths that represent a symplectic basis of the homology $H_1(X'; \mathbb{Z})$, and such that each path does not pass trough any singularity. We can assume that the $a_i,b_i$ are parametrized by the arc length. For each $a_i$ (resp. $b_i$), we define $\textrm{ind}(a_i)$  (resp. $\textrm{ind}(b_i)$) to be the index of the map $\mathbb{S}^1\to \mathbb{S}^1$, $t\mapsto a_i'(t)$ (resp. $t\mapsto b_i'(t))$). Then, the \emph{parity of spin structure} of $X'$ is defined by the following formula:
$$
\sum_{i=1}^{g'} (\textrm{ind}(a_i)+1)(\textrm{ind}(b_i)+1) \mod 2
$$ 
The result does not depend on the choice of the symplectic basis and is therefore an invariant of connected components of the strata of the moduli space of Abelian differentials (see \cite{KoZo}).

In the definition of the invariant for $\mathcal{C}(\mathcal{F}_{comp})$, we successively glue together some pairs of odd degree singularities. We can also glue all pairs except one and therefore we can assume that there is only one pair $(P_1,P_2)$ of odd degree singularities, of degree $k_1$ and $k_2$ respectively, on the surface. 

We present such surface $X$  as coming from the Veech construction with parameters $(\pi,\lambda,\tau)$. Let $g$ be the genus of this surface. As in Figure~\ref{curvpoly}, we have a pair $l_1^+,l_2^-$  of positive/negative vertical separatrix, we choose a path $\gamma$ transverse to the horizontal foliation. There exists a collection of closed  paths $a_1,b_1,\dots ,a_g,b_g$ that do not intersect $\gamma$  and that represent a symplectic basis of the homology $H_1(X,\mathbb{Z})$. Let also be $a_0$ a small circle around the singularity $P_1$.

When doing the parallelogram construction with parameters $l_1,l_2$  using $\gamma$, the closed paths $a_i,b_i$ persists and also the path $a_0$. Considering a path isometric to $\gamma$ inside the parallelogram, one obtains a closed path $b_0$, that intersect $a_0$ only once, and that does not intersect $a_i,b_i$ for all 
$i\in \{1,\dots ,g\}$. Hence, one gets  symplectic basis on the homology of the newly built surface, that can be used to compute the corresponding parity of spin structure. Here, as we will see later, the only relevant data are the indices of $a_0$ and $b_0$. We clearly have:
\begin{itemize}
\item $\textrm{ind}(a_0)=k_1+1 \mod 2=0 \mod 2$.
\item $\textrm{ind}(b_0)=0$.  
\end{itemize}

\begin{figure}[htb]
\begin{center}
\labellist
\tiny 
\hair 2pt
\pinlabel $l_1^+$ at 325 280   \pinlabel $l_2^-$ at 270 360   \pinlabel $1$ at 235 370  \pinlabel $2$ at 310 390
\pinlabel $3$ at 350 400     \pinlabel $4$ at 410 410    \pinlabel $5$ at 460 370  \pinlabel $5$ at 240 300
\pinlabel $4$ at 285 265   \pinlabel $3$ at 345 270   \pinlabel $2$ at 390 285   \pinlabel $1$ at 465 305

\pinlabel $5$ at 40 60   \pinlabel $4$ at 80 20   \pinlabel $0$ at 125 -10  \pinlabel $3$ at 175 25
\pinlabel $2$ at 220 45     \pinlabel $1$ at 290 60   \pinlabel $1$ at 35 125   \pinlabel $0$ at 90 160  
 \pinlabel $2$ at 135 145   \pinlabel $3$ at 175 155   \pinlabel $4$ at 240 160 \pinlabel $5$ at 280 125

\pinlabel $5$ at 400 60 \pinlabel $4$ at 445 25   \pinlabel $0$ at 490 -10  \pinlabel $3$ at 535 25
\pinlabel $2$ at 580 50     \pinlabel $1$ at 655 60   \pinlabel $1$ at 405 125   \pinlabel $0$ at 560 195  
 \pinlabel $2$ at 475 145   \pinlabel $3$ at 515 155   \pinlabel $4$ at 600 170 \pinlabel $5$ at 650 130

\pinlabel $l_3^-$ at 400 380
\pinlabel $\gamma$ at 300 330
\pinlabel $a_0$ at 400 330
\pinlabel $b_0$ at 115 80
\pinlabel $a_0$ at 225 80
\pinlabel $b_0'$ at 480 80
\pinlabel $a_0$ at 590 80
\endlabellist
\includegraphics[scale=0.5]{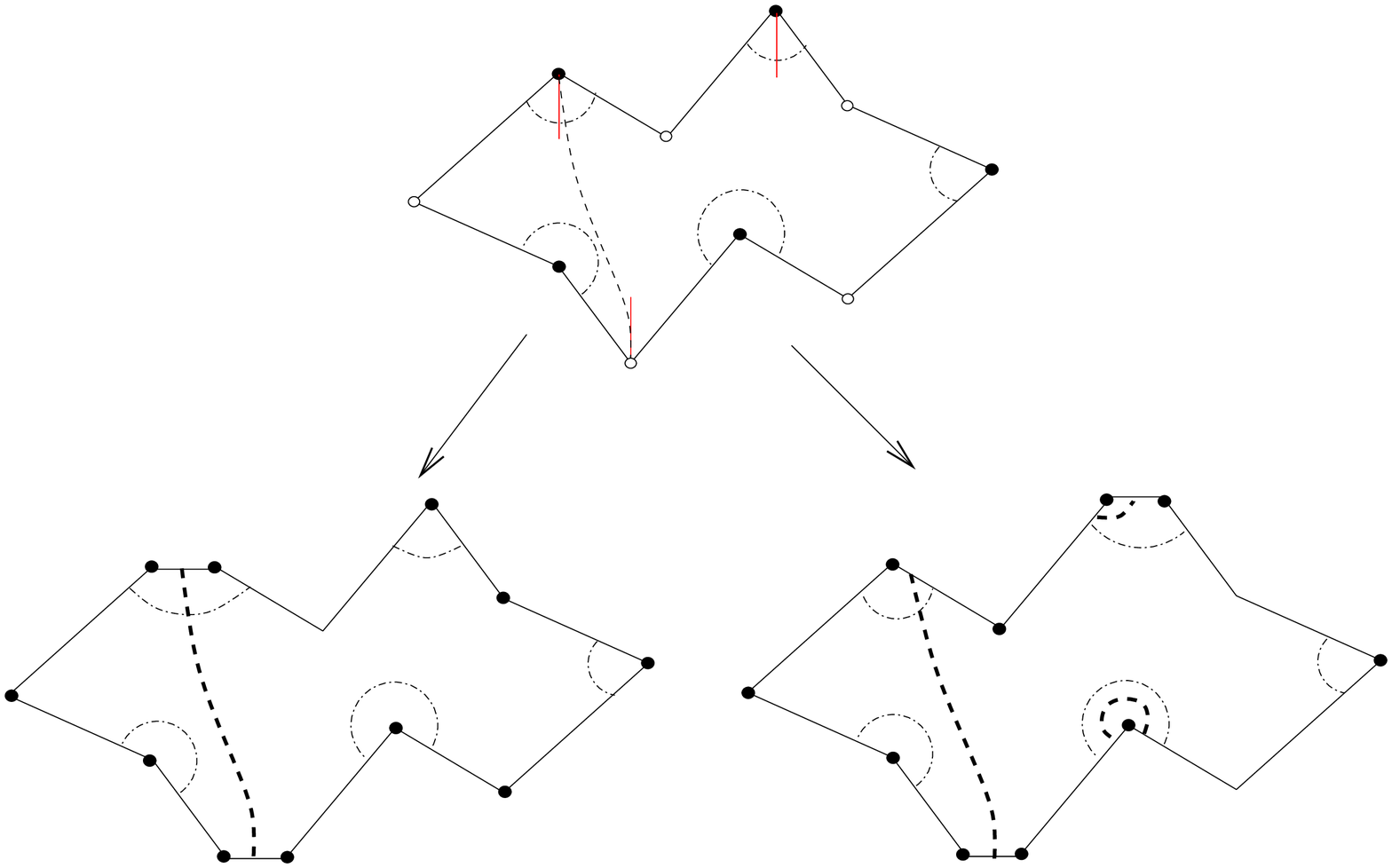}
\caption{Building two surfaces with different spin structure}
\label{fig:base:spin}
\end{center}
\end{figure}

Now we start again from the surface $X$ and replace the separatrix $l_2$  by the separatrix $l_3$, obtained by rotating $l_2$ by the angle $2\pi$. Then, we do the parallelogram construction with parameters  $(l_1,l_3)$. We  consider the following symplectic basis on the resulting flat surface:
\begin{itemize}
\item The path $a_i,b_i$, for $i\in \{1\dots g\}$  which persist under this construction.
\item The path $a_0$, which also persists under this construction.
\item A path $b_0'$  obtained as in Figure~\ref{fig:base:spin}
\end{itemize}

The indices of $a_0,a_1,\dots ,a_g$ and of $b_1,\dots,b_g $ are the same as previously, but $\textrm{ind}(b_0')=1$.

Since, $\textrm{ind}(a_0)+1=1 \mod 2$, the surface obtained from $(l_1,l_3)$ has a different parity of spin structure as the one obtained from $(l_1,l_2)$. Hence,  the two corresponding flat surfaces are in different connected components of the moduli space of Abelian differentials. This proves that  $\mathcal{C}(\mathcal{F}_{comp})$ has at least $2$ connected components.   

\end{proof}

%%%%%%%%%%%%%%%%%%%%%%%%%%%%%%%%%%%%%%%%%%
%%%%%%%%%%%%%%%%%%%%%%%%%%%%%%%%%%%%%%%%%%
\section{Number of connected components of $\mathcal{C}(\mathcal{F}_{comp})$}\label{surg} 
In the previous section, we have used topological invariants to find lower bounds on the number of connected components of $\mathcal{C}(\mathcal{F}_{comp})$. Here, we show that they are the exact values.
\subsection{Three elementary surgeries.}
Here we describe some elementary closed paths in $\mathcal{C}$ that lift to unclosed paths in $\mathcal{C}(\mathcal{F}_{comp})$.

Recall that a saddle connection $\gamma$ joining two distinct singularities is \emph{simple} if there exists no other saddle connection homologous to $\gamma$. In particular, it means that up to a small deformation of the surface in the ambient stratum, there is no other saddle connection in the surface parallel to $\gamma$. Then, deforming suitably the surface with the Teichmüller geodesic flow (see \cite{B2,Boissy:rc} for instance), one gets a surface for which the saddle connection corresponding to $\gamma$ is very short compared to the other ones. Then, one can show that such surface is   obtained by the \emph{breaking up a zero} surgery (see \cite{EMZ}). We give a short description of this surgery. 

\subsubsection*{Breaking up a singularity.}

\begin{figure}[htb]
\begin{center}
\labellist
\tiny 
\hair 2pt
\pinlabel $\varepsilon$ at 35 55
\pinlabel $\varepsilon$ at 35 80
\pinlabel $\varepsilon$ at 100 55
\pinlabel $\varepsilon$ at 100 80
\pinlabel $\varepsilon$ at 60 20
\pinlabel $\varepsilon$ at 60 100

\pinlabel $\varepsilon-\delta$ at 320 55
\pinlabel $\varepsilon-\delta$ at 320 80
\pinlabel $\varepsilon-\delta$ at 385 55
\pinlabel $\varepsilon-\delta$ at 385 85
\pinlabel $\varepsilon+\delta$ at 370 20
\pinlabel $\varepsilon+\delta$ at 370 115
\pinlabel $6\pi$  at 120 0
\pinlabel $4\pi+4\pi$  at 400 0
\pinlabel $\partial V_\varepsilon$ at 280 145 
\endlabellist
\includegraphics[scale=0.6]{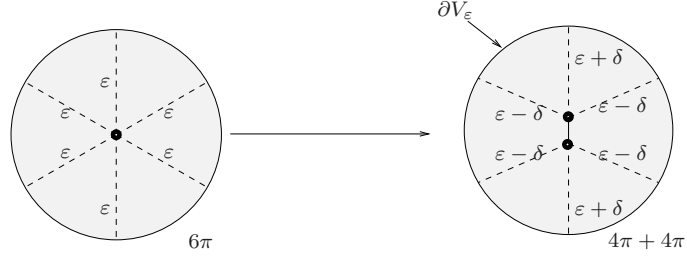}
\caption{Local surgery that break a zero of degree $k_1+k_2$ into two zeroes of degree $k_1$ and $k_2$ respectively.}
\label{break:zero}
\end{center}
\end{figure}

Let $k_1,k_2$ be the degree of the zeroes that are the endpoints of $\gamma$. We start from a zero $P$  of degree $k_1+k_2$. The neighborhood $V_\varepsilon=\{x\in X,  d(x,P)\leq \varepsilon\}$ of this conical singularity is obtained by considering $2(k_1+k_2)+2$ Euclidean half disks of radii $r$  and gluing each half side of them to another one in a cyclic order. We can break the zero into two smaller one by changing continuously the way they are glued to each other as in Figure~\ref{break:zero}. Note that in this surgery, the metric is not modified outside $V_\varepsilon$ . In particular, the boundary $\partial V_\varepsilon$   is isometric to (a connected  covering) of an Euclidean circle. Note that is this construction, we can ``rotate'' the two singularities by an angle $\theta$ by cutting the surface along $\partial V_\varepsilon$, rotating $V_\varepsilon$ by an angle $\theta$ and regluing it.
\medskip

\begin{move} \label{move1}
Let $X$ be a framed surface and let $P_1,P_2$ be two distinct singularities of degree $k$, joined by a simple saddle connection $\gamma$. We deform slightly the surface so that no saddle connection is parallel to $\gamma$. Then, using the Teichmüller geodesic flow, we contract the saddle connection $\gamma$  until it is very small compared to any other saddle connection. So the new surface $X'$ is obtained by breaking a zero of degree $2k$ into two zeroes $P_1'$ and $P_2'$ of degree $k$. Now we continuously rotate these two zeroes by the angle $\theta=(2k+1)\pi$. The resulting unframed surface is the same as $X'$,  but this procedure interchanges $P_1'$ and $P_2'$. Then, we come back to the initial surface $X$, but the labeled zeroes $P_1$ and $P_2$ have been interchanged. The labels on the separatrices adjacent to the other singularities have not changed.
\end{move}

The projection of this move in $\mathcal{C}$ is a closed path. This move in $\mathcal{C}(\mathcal{F}_{comp})$  interchanges $P_1$ and $P_2$, and fixes the separatrices adjacent to the other singularities.

The idea of this previous move is, as we will see, to authorize us to do any degree preserving permutation on the set of simply marked singularities. This explains the terms $n_i!$ in the formula of Theorem~\ref{MT}. Now the next two moves will fix the labeled  singularities and change marked the outgoing separatrices.

\begin{move}\label{move2}
Let $X$ be a framed surface and let $P_1,P_2$ be two distinct singularities of degree $k_1$ and $k_2$ respectively, joined by a simple saddle connection $\gamma$. Here we do not assume that $k_1=k_2$. We perform the same as in Move~\ref{move1}, but we turn $P_1'$ and $P_2'$ by $(k_1+k_2+1)2\pi$ instead.   
\end{move}

This move clearly corresponds to a closed path in $\mathcal{C}$. It also preserves $P_1,P_2$ pointwise. Let us look how changes the marked separatrices. For this, we can fix once for all a marked separatrix for all the singularities. Then, for a singularity of degree $k$, we can identify the set of corresponding horizontal separatrices to $\mathbb{Z}/(k+1)\mathbb{Z}$ by ordering them counterclockwise. We have the following lemma.

\begin{lemma}
Let $l_1\in \mathbb{Z}/(k_1+1)\mathbb{Z}$ and let $l_2\in \mathbb{Z}/(k_2+1)\mathbb{Z}$ be the separatrices associated to $P_1$ and $P_2$. Then, Move~\ref{move2} acts on the set of separatrices in the following way:
\begin{itemize}
\item  $l_1$ becomes $l_1-k_2 \mod k_1+1$ 
\item $l_2$ becomes  $l_2-k_1 \mod k_2+1$ 
\item  All the other labeled separatrices remain unchanged
\end{itemize}
\end{lemma}

\begin{proof}
Note that it is enough to prove this lemma in the case when $P_1,P_2$ is obtained after breaking up a singularity. The last statement of the lemma is obvious by construction: we do not change the metric outside a small neighborhood of $\gamma$.

Now we look at the surgery, keeping track of the labeled separatrices. When turning continuously the set $V_\varepsilon$ by an angle $\theta$, one must simultaneously change the separatrices by an angle $-\theta$, so that they stay horizontal. So at the end, they have moved by the angle $-(k_1+k_2+1)2\pi$ each, so for $i=1,2$,  $l_i$ is replaced by $l_i-(k_1+k_2+1)$ modulo $k_i+1$, which gives the result.
\end{proof}

Note that this move is especially useful when $k_1=k_2$. Then, the corresponding transformation is $(l_1,l_2)\mapsto (l_1+1,l_2+1)$.
\bigskip

Before describing the last move, we first describe a surgery which is analogous to the one presented in \cite{KoZo}. 

\subsubsection*{Bubbling $r$  handles:} \label{sec:bubble:handles}
We start from a singularity of degree $p\in \{0,1\}$. Let us consider a small polygonal line $L$ with no self intersection starting from the singularity. Let $r$ be the number of segments $s_1,\dots ,s_r$ of  $L$.  We consider $r$ parallelograms, each one having a pair on sides parallel to one of the $s_i$. Then, we cut the surface along each $s_i$ and paste in the corresponding parallelogram, and we glue by translation each remaining opposite sides of each parallelogram. We obtain a translation surface of genus $ g(X)+r$, and the degree p singularity have been replaced by a degree $p+2r$ singularity. Note that this surgery can be performed without changing the metric outside a small neighborhood of the singularity of degree $p$. Note that we can ``rotate'' the construction in the following way: the surgery is performed inside a $\varepsilon$ neighborhood $V_\varepsilon$ of the initial singularity of degree $p$. The boundary $\partial V_\varepsilon$ remains a metric covering of a euclidean circle after bubbling the handles. Now we can cut the surface $X$ along this circle and reglue it after a rotation by $\theta$.

Now we can describe the last move.
\begin{move}\label{move3}
Assume that the translation surface $X$ was obtained after bubbling $r$ handles and let $P$ be the corresponding singularity. We continuously rotate the construction as explained previously, by a angle of $(p+2r+1)2\pi$. 
\end{move}
  
As before, the underlying surface in $\mathcal{C}$  is the same after Move~\ref{move3} and any separatrix that does not correspond to the singularity $P$ remains unchanged. 

\begin{lemma}
Let $l\in \mathbb{Z}/(p+2r+1)\mathbb{Z}$ be the separatrix corresponding to $P$. Then, Move~\ref{move3} changes $l$ in the following way:
\begin{itemize}
\item If $p=0$, $l$ is replaced by $l-1$.
\item If $p=1$, $l$ is replaced by $l-2$.    
\end{itemize}
\end{lemma}

\begin{proof}
It is easy to see that, as in the case of Move~\ref{move2}, a marked separatrix attached to the singularity is changed by the transformation $l\mapsto l-(p+1)\mod p+2r+1$, and the separatrices associated to the other singularities remain unchanged.
\end{proof}

In particular if $p=0$, we reach all possible separatrices adjacent $P$   in this way.  If $p=1$, then $p+2r+1$ is even, and we reach only half of the separatrices adjacent to $P$ in this way. 

\subsection{Generating the monodromy group.}

\begin{proposition}\label{prop:monod:1}
Assume that $\mathcal{C}$ contains nonhyperelliptic surfaces, then the following holds:
\begin{itemize}
\item The set $\mathcal{C}(\mathcal{F}_{comp})$ is connected if all the singularities have even degree
\item The set $\mathcal{C}(\mathcal{F}_{comp})$ has two connected components otherwise.
\end{itemize}
\end{proposition}

We start with the following lemma. 
\begin{lemma}
Let $\mathcal{C}$ be a connected component of $\mathcal{H}(k_1^{n_1},\ldots ,k_r^{n_r})$.  Choose an ordering on the set with multiplicities $\{k_1,\dots ,k_1,\dots ,k_r,\dots  ,k_r\}$. There exists $X\in \mathcal{C}$ and a polygonal line in $X$ that consists of simple saddle connections and that joins all the singularities of $X$  in that order.
\end{lemma}

\begin{proof}
The proof is the same as  the proof of Proposition~3.5 in \cite{Boissy:rc}. 
\end{proof}

\begin{proof}[Proof of Proposition~\ref{prop:monod:1}]
We first assume that there exists  odd degree singularities in the underlying stratum. Since we assume that it is not a hyperelliptic stratum, it is connected (see \cite{KoZo}). We write this stratum as $\mathcal{H}(k_1^{n_1},\dots ,k_{s}^{n_s}, (2k'_1)^{\beta_1},\dots ,(2k'_{s'})^{\beta_{s'}})$ , with $s+s'=r$, and $k_1,\dots ,k_{s}$ are odd.

Now we start from a surface in $\mathcal{H}(k_1^{n_1-1},1,\dots ,k_{s}^{n_s-1},1)$ (i.e. we don't take the even degree singularities and we replace one singularity of each odd degree by a singularity of degree one). From the previous lemma, we can assume that there is a polygonal path of simple saddle connections which has no self intersection and that joins successively all the singularities in the following order: 
\begin{itemize}
\item first the singularities  of degree $k_1$,
\item then, a singularity of degree 1,
\item then, the singularities of degree $k_2$,
\item then, a singularity of degree 1,
\item and so on \dots  
\end{itemize}

Now for each singularity of degree $1$ that ends a group of singularities of degree $k_i$, we bubble $(k_i-1)/2$ handles as in section~\ref{sec:bubble:handles}. This replace the singularity of degree $1$ by a singularity of degree $k_i$. Note that the polygonal path  of simple saddle connections persists under this surgery. We will denote by $\gamma$ this polygonal line. 

Now for each $i\in \{1,\dots ,s'\}$, we consider a polygonal line $\gamma_i$ joining $\beta_i $ regular points, such that the paths $\gamma,(\gamma_i)_i$ have no intersection points. Then, for each vertex, we bubble $n_i$ handles. We obtain $s'$ chains of simple saddle connections that join each collection of singularities of degree $2k'_i$. The resulting surface is therefore in $\mathcal{H}(k_1^{n_1},\dots ,k_{s}^{n_s}, (2k'_1)^{\beta_1},\dots ,(2k'_{s'})^{\beta_{s'}})$, which is the stra\-tum that we stu\-dy.

Now using Move~1, we see that for any polygonal line of simple saddle connections joining singularities with the same degree, we can perform any transposition of two consecutive singularities.  Hence we can arbitrarily permute the singularities sharing the same degree.

Using Move~3, we see that we can reach any choice of separatrices for the even degree singularities.

Now we consider the chain $\gamma$ of simple saddle connections joining all the separatrices of odd degree, that was constructed before. The first  $n_1$ vertices of $\gamma$ makes a chain of singularities of degree $k_1$ . Let us name the singularities $P_{1,1},\dots ,P_{1,n_1}$ according to the order given by the polygonal path $\gamma$. If $n_1>1$, we reach any choice of separatrices for $P_{1,1},\dots ,P_{1,n_1-1}$ by applying successively Move~2 on the pairs $(P_{1,i},P_{1,i+1})$ for $i\in \{1,\dots ,n_1-1\}$, in order to choose arbitrarily a labeled separatrix of $P_{1,i}$. Note that once this is done for some $i$, the next moves don't change the marked separatrix corresponding to $P_{1,i}$. Then, for the singularity $P_{1,n_1}$, we use Move~3 to rotate the corresponding separatrix by any even number (recall that the set of outgoing separatrices corresponding to a singularity of degree $k$ is naturally identified with $\mathbb{Z}/(k+1)\mathbb{Z})$). If $P_{1,n_1}$ is not the end of $\gamma$, \emph{i.e.} the polygonal line $\gamma$ continues to some other (odd) degree singularity $P_{2,1}$, then Move~2 on the pair $P_{1,n_1},P_{2,1}$ will act on the marked separatrice of $P_{1,n_1}$  as $l_1\mapsto l-k_2$. Hence,  it will be changed by a \emph{odd} number, so in combination with Move~3, we obtain all possible choices.

If we iterate this procedure until the last singularity of degree $k_s$, we see that we can reach any choice of separatrices for the singularities of odd degree, except the last one of the chain where we obtain only half of the possibilities. This proves the proposition in the case when there exists odd singularities.

\medskip
If there does not exists any singularity of odd degree, the procedure described above works (with $\gamma=\emptyset$) as soon as we can find a surface like above in the connected component that we study. But in this case, the corresponding stratum of translation surfaces is not connected. 

Consider a translation surface obtained from a torus with the "bubbling $r$ handles" construction. We can easily show that in this case, each singularity contribute to zero to the spin structure. See Figure~\ref{spin:bubble1}. Hence the resulting parity of spin structure is the same as for the flat torus, which is odd.
\begin{figure}[htb]
\begin{center}
\labellist
\tiny 
\hair 2pt
\pinlabel 1 at 150 220   \pinlabel 2 at 170 220
\pinlabel 1 at 420 290   \pinlabel 2 at 540 290 
\pinlabel 3 at 140 160   \pinlabel 4 at 160 160
\pinlabel 3 at 410 200   \pinlabel 4 at 515 200
\pinlabel 5 at 130 100   \pinlabel 6 at 150 100
\pinlabel 5 at 410 120 \pinlabel 6 at 515 120
\pinlabel $a_1$ at 220 240 \pinlabel $b_1$ at 500 305
\pinlabel $a_2$ at 250 240 \pinlabel $b_2$ at 475 220
\pinlabel $a_3$ at 300 240 \pinlabel $b_3$ at 460 135
\pinlabel $a_1$ at 450 290 
\pinlabel $a_2$ at 440 200 \pinlabel $a_3$ at 440 115
\endlabellist
\includegraphics[scale=0.5]{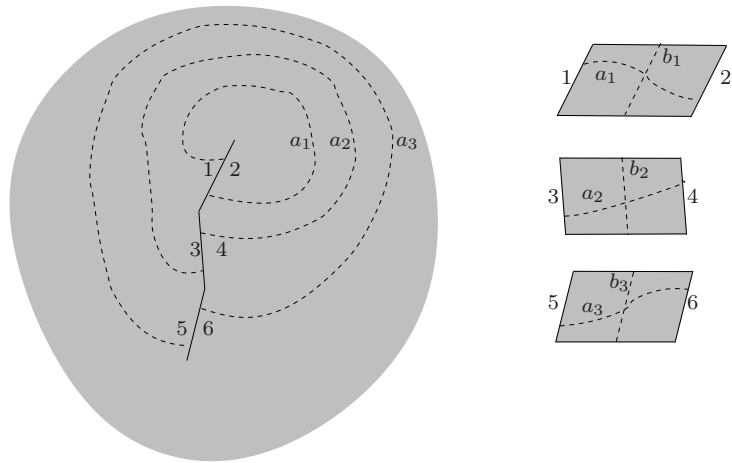}
\caption{The ``bubbling $r$-handles standard construction. We have $\textrm{ind}(a_i)=1$, so the collection $(a_i,b_i)_{i\in \{1\dots \beta\}}$ contributes to $0$ for the spin structure.}
\label{spin:bubble1}
\end{center}
\end{figure}

If there exists a singularity of degree $k\geq 4$. It is easy to see that one can slightly change the construction to make this singularity contribute to 1 to the spin structure, and obtain a surface with even spin structure (see~Figure~\ref{spin:bubble2})

\begin{figure}[htb]
\begin{center}
\labellist
\tiny 
\hair 2pt
\pinlabel 1 at 148 212   \pinlabel 2 at 175 170
\pinlabel 1 at 420 290   \pinlabel 2 at 540 290 
\pinlabel 3 at 135 165   \pinlabel 4 at 180 210
\pinlabel 3 at 410 200   \pinlabel 4 at 515 200
\pinlabel 5 at 130 100   \pinlabel 6 at 150 100
\pinlabel 5 at 410 120 \pinlabel 6 at 515 120
\pinlabel $a_1$ at 235 250 \pinlabel $b_1$ at 500 305
\pinlabel $a_2$ at 190 260 \pinlabel $b_2$ at 475 220
\pinlabel $a_3$ at 300 240 \pinlabel $b_3$ at 450 135
\pinlabel $a_1$ at 450 285 \pinlabel $a_2$ at 450 325
\pinlabel $a_2$ at 440 205  \pinlabel $a_3$ at 440 115
\pinlabel $a_2$ at 105 170

\endlabellist
\includegraphics[scale=0.5]{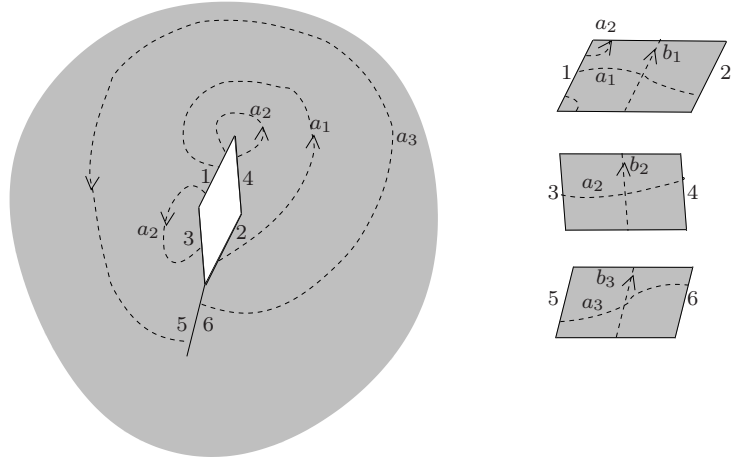}
\caption{A slight change in the $r$-handle construction, for $r\geq 2$ changes the spin structure, since in this case, $\textrm{Ind}(a_2)=0 \mod 2$.}
\label{spin:bubble2}
\end{center}
\end{figure}

The last remaining case to see is when all the singularities are of degree $2$, and the parity of spin structure is even. If there are exactly $2$ singularities, then the connected component is $\mathcal{H}^{even}(2,2)=\mathcal{H}^{hyp}(2,2)$, which is a hyperelliptic case. If there are at least $3$ singularities. We can find a surface with a chain of simple saddle connections joining all the singularities, and such that the last element of the chain is obtained by the "bubbling a handle" construction. Then, combining Move~2 along the chain, and Move~3 at the end of the chain, we obtain that $\mathcal{C}(\mathcal{F}_{comp})$ is connected. This concludes the proof of the proposition.
\end{proof}

\begin{proof}[Proof of Theorem~\ref{th:cc:Cfr}]
The nonhyperelliptic case is given by Proposition~\ref{prop:monod:1}. For the hyperelliptic case, the lower bound on the number of connected components is given by Proposition~\ref{low:bound:hyp}, and the upper bound is easy and left to the reader.
\end{proof}

\begin{proof}[Proof of Theorem~\ref{MT}]
Recall that we denote by $\mathcal{H}(k_1^{n_1},\dots ,k_r^{n_r})$ the ambient stratum of the moduli space of Abelian differentials. 
The degree of the covering $\mathcal{C}(\mathcal{F}_{comp})\to \mathcal{C}$, restricted to a connected component of $\mathcal{C}(\mathcal{F}_{comp})$, is clearly $\frac{\Pi_{i=1}^r n_i! (k_i+1)^{\alpha_i}}{c}$, were $c$ is the number of connected component of $\mathcal{C}(\mathcal{F}_{comp})$ and is given by Theorem~\ref{th:cc:Cfr}. 

Let $k$ be the degree of the marked singularity associated the Rauzy class $R$,  and let $n$ be the number of singularities of degree $k$. The set $\mathcal{C}_{k}$ is connected and the degree of the projection $\mathcal{C}_{k}\to \mathcal{C}$ is $n (k+1)$. Hence we have:

$$
\frac{|R_{lab}|}{|R|}=\frac{\Pi_{i=1}^r n_i! (k_i+1)^{\alpha_i}}{c.(k+1).n}
$$
Which gives Theorem~\ref{MT}.
\end{proof}

\appendix 
\section{Rauzy classes for quadratic differentials}

\emph{Half-translation surfaces} are a natural generalization of translation surfaces. They are surfaces with an atlas such that the changes of coordinates are not only translations, but can also be half-turns. They corresponds to Riemann surfaces with \emph{quadratic differentials}.

Danthony and Noguiera have generalized interval exchange transformations and Rauzy induction to describe first return maps of nonoriented measured foliations on transverse segment. One gets \emph{linear involutions}, see \cite{DaNo}.

The relation between quadratic differentials and linear involutions was described by the author and Lanneau in \cite{BL09}. 

One can wonder if there is a analogous result for Rauzy classes appearing in this context. There doesn't seem to be a natural relation between "labeled generalized permutations" and framed half-translation surfaces. In particular, there is no ``quadratic'' equivalent of Lemma~\ref{Rlab:to:Hsat}.

Numerical experiments on SAGE, suggest that the ratio between a labeled Rauzy class and its corresponding reduced one is always $n!$ or $\frac{n!}{2}$, were $n$ is the number of underlying intervals, which is generally much more than the Abelian case. In particular, using Rauzy induction and labelling the intervals, these numerical experiments suggests that one obtain either any renumbering of the intervals, or any even renumbering of the intervals depending on the stratum.

One can also look at \emph{extended Rauzy classes}, where we add Rauzy moves that correspond to cutting on the left of the interval (see \cite{KoZo}). In this case, numerical experiments suggests that the ratio is also $n!$ or $n!/2$ depending on the stratum.

\end{document}